\documentclass{article}

\usepackage{amsthm}
\usepackage{amsmath, amssymb}
\usepackage{hyperref}
\usepackage{tikz}
\usetikzlibrary{matrix}
\usepackage[matrix, arrow, curve]{xy}

\newtheorem{Thm}{Theorem}[section]

\newtheorem{Lem}[Thm]{Lemma}
\newtheorem{Prop}[Thm]{Proposition}

\newtheorem{Rem}[Thm]{Remark}

\title{Symplectic orbits of unimodular rows}

\author{Tariq Syed  \\
	Department of Mathematics\\
	University of Southern California\\
	3620 South Vermont Ave.\\
	Los Angeles, CA 90089, United States\\
	tariq.syed@gmx.de}

\date{\today}
\begin{document}

\maketitle

\begin{abstract}
For a smooth affine algebra $R$ of dimension $d \geq 3$ over a field $k$ and an invertible alternating matrix $\chi$ of rank $2n$, the group $Sp(\chi)$ of invertible matrices of rank $2n$ over $R$ which are symplectic with respect to $\chi$ acts on the right on the set $Um_{2n}(R)$ of unimodular rows of length $2n$ over $R$. In this paper, we prove that $Sp(\chi)$ acts transitively on $Um_{2n}(R)$ if $k$ is algebraically closed, $d! \in k^{\times}$ and $2n \geq d$.
\end{abstract}

\tableofcontents

\section{Introduction}
Let $R$ be a commutative ring. For $n \geq 1$, we let $Um_{n}(R)$ denote the set of unimodular rows of length $n$ over $R$, i.e., the set of row vectors $(a_{1},...,a_{n})$ with $a_{i} \in R$, $1 \leq i \leq n$, such that $\langle a_{1},...,a_{n} \rangle = R$. The group $GL_{n}(R)$ of invertible matrices of rank $n$ over $R$ and hence any subgroup of $GL_{n}(R)$ act on the right on $Um_{n}(R)$; in particular, if $\chi$ is an invertible alternating matrix of rank $2n$, then the group $Sp(\chi)$ of invertible matrices of rank $2n$ which are symplectic with respect to $\chi$ acts on the right on $Um_{2n}(R)$.\\
Now let $R$ be a smooth affine algebra of dimension $d \geq 3$ over a field $k$. It follows from \cite[Chapter IV, Theorem 3.4]{HB} that $E_{n}(R)$ acts transitively on $Um_{n}(R)$ if $n \geq d+2$. The main result in \cite{S1} implies that $SL_{n}(R)$ acts transitively on $Um_{n}(R)$ if $k$ is algebraically closed and $n = d+1$; similarly, the main result in \cite{S4} implies that $SL_{n}(R)$ acts transitively on $Um_{n}(R)$ if $c.d.(k) \leq 1$, $d! \in k^{\times}$ and $n = d+1$. Furthermore, it was proven in \cite{FRS} that $SL_{n}(R)$ acts transitively on $Um_{n}(R)$ if $k$ is algebraically closed, $(d-1)! \in k^{\times}$ and $n = d$.\\
Now let $n \geq 1$ and let $\chi$ be an invertible alternating matrix of rank $2n$. It follows from \cite[Lemma 5.5]{SV} that $Sp(\chi)$ acts transitively on $Um_{2n}(R)$ if $2n \geq d+2$. In this paper, we investigate the transitivity of the groups $Sp(\chi)$ on $Um_{2n}(R)$ whenever $2n = d$ or $2n = d+1$. We prove the following result (cf. Theorem \ref{Trans1} and Theorem \ref{Trans2} in the text):\\

\textbf{Theorem.} Let $R$ be a smooth affine algebra of dimension $d \geq 3$ over a field $k$ with $d! \in k^{\times}$. For $n \geq 1$, let $\chi$ be an invertible alternating matrix of rank $2n$. Assume that one of the following conditions is satisfied:
\begin{itemize}
\item $2n = d+1$, $d+1 \equiv 2~mod~4$ and $c.d.(k) \leq 1$
\item $2n = d+1$, $d+1 \equiv 0~mod~4$, $c.d.(k) \leq 1$ and $k$ is perfect
\item $2n = d$ and $k$ is algebraically closed
\end{itemize}
Then $Sp(\chi)$ acts transitively on $Um_{2n}(R)$.
\\\\
If $\chi = \psi_{2n}$, then $Sp (\chi)$ is just the usual symplectic group $Sp_{2n}(R)$; in the situation of the theorem above, it was already established in \cite{G} and \cite{Sy2} that $Sp_{d} (R)$ acts transitively on $Um_{d}(R)$ if $d \geq 4$ is even and $k$ is algebraically closed. Similar results in odd dimensions were proven in \cite{BCR}.\\
The transitivity results above have immediate implications on the injectivity of the generalized Vaserstein symbol introduced and studied in \cite{Sy} and \cite{Sy2}: Recall that if $R$ is a commutative ring, $P_{0}$ is a projective $R$-module of rank $2$ with a fixed trivialization $\theta_{0}: R \xrightarrow{\cong} \det(P_{0})$ of its determinant, the author has introduced a generalized Vaserstein symbol

\begin{center}
$V_{\theta_{0}}: Um (P_{0} \oplus R)/E (P_{0} \oplus R) \rightarrow \tilde{V}(R)$
\end{center}

associated to $P_{0}$ and $\theta_{0}$. Here $Um (P_{0} \oplus R)$ denotes the set of epimorphisms $P_{0} \oplus R \rightarrow R$ and $E (P_{0} \oplus R)$ is the subgroup of the group $Aut (P_{0} \oplus R)$ of automorphisms of $P_{0} \oplus R$ generated by elementary automorphisms. The group $\tilde{V} (R)$ is canonically isomorphic to the elementary symplectic Witt group $W_E (R)$ (cf. \cite[\S 3]{SV}). The generalized Vaserstein symbol descends to a map

\begin{center}
$V_{\theta_{0}}: Um (P_{0} \oplus R)/SL (P_{0} \oplus R) \rightarrow \tilde{V}_{SL}(R)$,
\end{center}

where $SL (P_{0} \oplus R)$ is the group of automorphisms of $P_{0} \oplus R$ of determinant $1$ and $\tilde{V}_{SL}(R)$ is the cokernel of a hyperbolic map $SK_{1}(R) \rightarrow \tilde{V}(R)$. Criteria for the injectivity and surjectivity of this map are given in \cite{Sy2}. As a corollary of the theorem above, we obtain the following result (cf. Theorem \ref{T3.9}):\\

\textbf{Theorem.} Let $R$ be a smooth affine algebra of dimension $d \geq 3$ over an infinite field $k$ with $d! \in k^{\times}$. Let $\theta_{0}: R \xrightarrow{\cong} \det (R^2)$ be a fixed trivialization of $\det (R^2)$. Assume that one of the following conditions is satisfied:
\begin{itemize}
\item $d=3$, $k$ is perfect and $c.d.(k) \leq 1$
\item $d=4$ and $k$ is algebraically closed
\end{itemize}
Then $V_{\theta_{0}}: Um_{3}(R)/SL_{3}(R) \xrightarrow{\cong} \tilde{V}_{SL}(R)$ is a bijection.
\\\\
The second case in the theorem strengthens the main result in \cite{Sy2}: In the second case of the theorem, the orbit space $Um_{3}(R)/SL_{3}(R)$ is trivial if and only if the group $\tilde{V}_{SL}(R)$ is trivial (cf. \cite[Theorem 3.19]{Sy2}).\\
In fact, if we consider $3$-dimensional affine algebras over a finite field $\mathbb{F}_{q}$, the criterion for the injectivity of the generalized Vaserstein symbol in \cite{Sy2} allows us to prove the more general result (cf. Theorem \ref{T3.11} in the text):\\

\textbf{Theorem.} Assume that $R$ is an affine algebra of dimension $d=3$ over a finite field $\mathbb{F}_{q}$. Let $P_{0}$ be a projective $R$-module of rank $2$ with a trivialization $\theta_{0}: R \xrightarrow{\cong} \det(P_{0})$ of its determinant. Then the generalized Vaserstein symbol descends to a bijection $V_{\theta_{0}}: Um (P_{0} \oplus R)/SL (P_{0} \oplus R) \xrightarrow{\cong} \tilde{V}_{SL} (R)$.\\\\
The organization of the paper is as follows: In Section \ref{2.1} we introduce the groups $W_E (R)$ and $W_{SL}(R)$ for a commutative ring $R$ and discuss the Karoubi periodicity sequence. Section \ref{2.2} serves as a brief introduction to $\mathbb{A}^{1}$-homotopy theory and Suslin matrices. In Section \ref{2.3} we recall the definition and properties of the generalized Vaserstein symbol. Finally, we prove the main results of this paper in Section \ref{3}.

\subsection*{Acknowledgements}
The author would like to thank Jean Fasel for many helpful discussions and for his support. Furthermore, he would like to thank Ravi Rao for very helpful comments. The author was funded by the Deutsche Forschungsgemeinschaft (DFG, German Research Foundation) - Project number 461453992.%Moreover, the author would like to thank Aravind Asok and Paul Arne {\O}stv{\ae}r for helpful comments.

\section{Preliminaries}\label{Preliminaries}\label{2}

Let $R$ be a commutative ring. We denote by $Um_{n} (R)$ the set of unimodular rows of length $n$ over $R$, i.e., row vectors $(a_{1},...,a_{n})$ with $a_{i} \in R$, $1 \leq i \leq n$, such that $\langle a_{1},...,a_{n} \rangle=R$; similarly, we denote by $Um_{n}^{t} (R)$ the set of unimodular columns of length $n$ over $R$, i.e., column vectors ${(a_{1},..,a_{n})}^{t}$ with $(a_{1},...,a_{n}) \in Um_{n}(R)$. We let $GL_{n} (R)$ be the group of invertible matrices of rank $n$ over $R$. Moreover, we let $SL_{n} (R)$ be its subgroup of matrices of determinant $1$ and $E_{n} (R)$ be its subgroup generated by elementary matrices. If $\chi$ is an alternating invertible matrix of rank $2n$, then we let $Sp (\chi)$ be the subgroup of $SL_{2n} (R)$ consisting of all the matrices which are symplectic with respect to $\chi$. The group $GL_{n}(R)$ and hence all its subgroups mentioned above act on the right on $Um_{n} (R)$ and on the left on $Um_{n}^{t}(R)$. We will usually denote by $\pi_{i,n}$ the unimodular row $(0,...,0,1,0,...,0)$ of length $n$ with $1$ in the $i$th slot and $0$'s elsewhere and we let $\pi_{n} = \pi_{n,n}$. Similarly, we let $e_{i,n} = \pi_{i,n}^{t} \in Um_{n}^{t}(R)$ and $e_{n} = \pi_{n}^{t} \in Um_{n}^{t}(R)$.\\
More generally, let $P$ be a finitely generated projective $R$-module. Then we denote by $Um (P)$ the set of epimorphisms $P \rightarrow R$ and by ${Um}^{t}(P)$ the set of unimodular elements of $P$. Furthermore, we denote by $Aut(P)$ the group of automorphisms of $P$ and by $SL (P)$ its subgroup of automorphisms of determinant $1$. Moreover, we let $E (P)$ denote the subgroup of $Aut (P)$ generated by transvections. If $\chi$ is a non-degenerate alternating form on $P$, then we denote by $Sp(\chi)$ the subgroup of $Aut(P)$ consisting of automorphisms of $P$ which are symplectic with respect to $\chi$.\\
If $P = R^n$ for $n \geq 1$, we can identify $Um (P)$ with $Um_{n}(R)$, $Um^{t}(P)$ with $Um_{n}^{t}(R)$, $Aut(P)$ with $GL_{n}(R)$, $SL(P)$ with $SL_{n}(R)$ and $E(P)$ with $E_{n}(R)$; for the last statement, see \cite{BBR}. The right action of $Aut(R^n)$ and its subgroups on $Um(R^n)$ then just corresponds to matrix multiplication; analogously, the left action of $Aut(R^n)$ and its subgroups on $Um^{t} (R^n)$ is also just given by matrix multiplication. Furthermore, we can also identify non-degenerate alternating forms on $P$ with invertible alternating matrices of rank $n$; in particular, our definition of $Sp(\chi)$ for a non-degenerate alternating form $\chi$ on $R^n$ then coincides under these identifications with the definition of $Sp(\chi)$ above if we interpret $\chi$ as an invertible alternating matrix of rank $n$.

\subsection{The groups $W_E (R)$ and $W_{SL}(R)$}\label{The groups $W_E (R)$ and $W_{SL}(R)$} \label{2.1}

Let $R$ be a commutative ring. For $n \geq 1$, we denote by $A_{2n}(R)$ the set of invertible alternating matrices of rank $2n$. We inductively define $\psi_{2n} \in  A_{2n} (R)$ by
 
\begin{center}
$\psi_2 =
\begin{pmatrix}
0 & 1 \\
- 1 & 0
\end{pmatrix}
$
\end{center}
 
\noindent and $\psi_{2n+2} = \psi_{2n} \perp \psi_2$. There is a canonical embedding $A_{2m} (R) \rightarrow A_{2n} (R)$, $M \mapsto M \perp \psi_{2n-2m}$ for any $m < n$. We let $A (R)$ be the direct limit of the sets $A_{2n} (R)$ under these embeddings. Furthermore, we call two alternating invertible matrices $M \in A_{2m} (R)$ and $N \in A_{2n} (R)$ equivalent, $M \sim N$, if there is an integer $s \geq 1$ and a matrix $E \in E_{2n+2m+2s}(R)$ such that

\begin{center}
$M \perp \psi_{2n+2s} = E^{t} (N \perp \psi_{2m+2s}) E$.
\end{center}

This relation defines an equivalence relation on $A(R)$ and the corresponding set of equivalence classes $A(R)/{\sim}$ is denoted $W'_E (R)$. Since
 
\begin{center}
$
\begin{pmatrix}
0 & id_{s} \\
id_{r} & 0
\end{pmatrix}
\in E_{r+s} (R)$
\end{center}
 
for even $rs$, it immediately follows that the orthogonal sum equips $W'_E (R)$ with the structure of an abelian monoid. It is proven in \cite[\S 3]{SV} that this abelian monoid is actually an abelian group: An inverse for an element of $W'_E (R)$ represented by a matrix $N \in A_{2n} (R)$ is given by the element represented by the matrix $\sigma_{2n} N^{-1} \sigma_{2n}$, where the matrices $\sigma_{2n}$ are inductively defined by
 
\begin{center}
$\sigma_2 =
\begin{pmatrix}
0 & 1 \\
1 & 0
\end{pmatrix}
$
\end{center}
 
\noindent and $\sigma_{2n+2} = \sigma_{2n} \perp \sigma_2$.
One can assign to any alternating invertible matrix $M$ an element $\mathit{Pf}(M)$ of $R^{\times}$ called the Pfaffian of $M$. The Pfaffian satisfies the following formulas for all integers $m,n \geq 1$:
 
\begin{itemize}
\item $\mathit{Pf}(M \perp N) = \mathit{Pf}(M) \mathit{Pf}(N)$ for all $M \in A_{2m} (R)$ and $N \in A_{2n} (R)$
\item $\mathit{Pf}(G^{t} N G) = \det (G) \mathit{Pf}(N)$ for all $G \in GL_{2n} (R)$ and $N \in A_{2n} (R)$
\item ${\mathit{Pf} (N)}^{2} = \det (N)$ for all $N \in A_{2n} (R)$
\item $\mathit{Pf} (\psi_{2n}) = 1$
\end{itemize}

\noindent It follows directly from these formulas that the Pfaffian determines a group homomorphism $\mathit{Pf}: W'_E (R) \rightarrow R^{\times}$; its kernel $W_E (R) := \ker (\mathit{Pf})$ is called the elementary symplectic Witt group of $R$. The homomorphism $\mathit{Pf}: W'_E (R) \rightarrow R^{\times}$ is split by the homomorphism $R^{\times} \rightarrow W'_E (R)$ which assigns to any $t \in R^{\times}$ the class in $W'_E (R)$ represented by the matrix

\begin{center}
$\begin{pmatrix}
0 & t \\
-t & 0
\end{pmatrix}
$.
\end{center}

Hence there is an isomorphism $W'_E (R) \cong W_E (R) \oplus R^{\times}$.\\
The group $W'_E (R)$ fits into an exact Karoubi periodicity sequence

\begin{center}
$K_{1}{Sp} (R) \xrightarrow{f} K_{1} (R) \xrightarrow{H} W'_{E} (R) \xrightarrow{\eta} K_{0}{Sp} (R) \xrightarrow{f'} K_{0} (R)$.
\end{center}

The homomorphisms $f$ and $f'$ are both forgetful homomorphisms, i.e., $f$ and $f'$ are induced by the obvious inclusions $Sp_{2n} (R) \rightarrow GL_{2n} (R)$ and the assignment $(P, \varphi) \mapsto P$ for any skew-symmetric space $(P, \varphi)$ respectively. Moreover, the homomorphism $K_{1} (R) \xrightarrow{H} W'_{E} (R)$ is induced by the assignment $M \mapsto M^{t} \psi_{2n} M$ for all $M \in GL_{2n} (R)$. Finally, the homomorphism $W'_{E} (R) \xrightarrow{\eta} K_{0}{Sp} (R)$ is induced by the assignment $M \mapsto [R^{2n}, M] - [R^{2n},\psi_{2n}]$ for all $M \in A_{2n} (R)$.\\% A direct proof that these assignments give well-defined homomorphisms and the sequence above is exact can be found in \cite[Section 2]{F}.\\
Since the image of $K_{1}{Sp} (R)$ under $f$ in $K_{1} (R)$ lies in $SK_{1} (R)$, one can rewrite the sequence above as

\begin{center}
$K_{1}{Sp} (R) \xrightarrow{f} SK_{1} (R) \xrightarrow{H} W_{E} (R) \xrightarrow{\eta} K_{0}{Sp} (R) \xrightarrow{f'} K_{0} (R)$.
\end{center}

The group $W_{SL}(R)$ is the cokernel of the homomorphism $SK_{1} (R) \xrightarrow{H} W_E (R)$. We now characterize invertible matrices $\varphi \in GL_{2n}(R)$ whose class in $K_{1}(R)$ lies in the kernel of $H$:

\begin{Lem}\label{L2.1}
Let $R$ be a commutative ring and let $\chi \in A_{2n} (R)$. If $\varphi \in GL_{2n} (R)$ such that its class $[\varphi] \in K_{1} (R)$ lies in $\ker (H)$, then there are $m \in \mathbb{N}$ and $\varphi' \in SL_{2n+2m} (R)$ such that $[\varphi] = [\varphi'] \in K_1 (R)$ and $\varphi'$ is symplectic with respect to $\chi \perp \psi_{2m}$.
\end{Lem}

\begin{proof}
By definition of $W'_E (R)$, one obtains $s \in \mathbb{N}$ and $\varphi_1 \in E_{2n+2s} (R)$ such that

\begin{center}
$\varphi_1^t ({\varphi}^{t} \perp id_{2s}) \psi_{2n+2s} (\varphi \perp id_{2s}) \varphi_1 = \psi_{2n+2s}$.
\end{center}

In particular, one obtains

\begin{center}
$(id_{2n} \perp \varphi_1^t) (id_{2n} \perp {\varphi}^{t} \perp id_{2s}) (\chi \perp \psi_{2n+2s}) (id_{2n} \perp \varphi \perp id_{2s}) (id_{2n} \perp \varphi_1) =$\\$\chi \perp \psi_{2n+2s}$
\end{center}

and, since $id_{2n} \perp \varphi \perp id_{2s} = (\varphi \perp id_{2n+2s}) (\varphi^{-1} \perp \varphi \perp id_{2s})$,

\begin{center}
$(id_{2n} \perp \varphi_1^t) {(\varphi^{-1} \perp \varphi \perp id_{2s})}^{t} {(\varphi \perp id_{2n+2s})}^{t} (\chi \perp \psi_{2n+2s}) (\varphi \perp id_{2n+2s})$\\$(\varphi^{-1} \perp \varphi \perp id_{2s})(id_{2n} \perp \varphi_1) = \chi \perp \psi_{2n+2s}$
\end{center}

We set $m = n+s$ and $\varphi' = (\varphi \perp id_{2n+2s})(\varphi^{-1} \perp \varphi \perp id_{2s})(id_{2n} \perp \varphi_1)$. Since $\varphi \perp {\varphi}^{-1} \in E_{4n} (R)$, this proves the statement.
\end{proof}

\subsection{Motivic homotopy theory and Suslin matrices}\label{Motivic homotopy theory and Suslin matrices}\label{2.2}

Let $k$ be a field. We consider the category $Sm_{k}$ of smooth separated schemes of finite type over $k$ and we let $Spc_{k} = \Delta^{op} Shv_{Nis} (Sm_{k})$ (resp. $Spc_{k,\bullet}$) be the category of (pointed) simplicial Nisnevich sheaves on $Sm_{k}$. Furthermore, we write $\mathcal{H}_{s} (k)$ (resp. $\mathcal{H}_{s,\bullet} (k)$) for the (pointed) Nisnevich simplicial homotopy category which can be obtained as the homotopy category of the injective local model structure on $Spc_{k}$ (resp. $Spc_{k,\bullet}$). Finally, we denote by $\mathcal{H} (k)$ (resp. $\mathcal{H}_{\bullet} (k)$) the unstable $\mathbb{A}^{1}$-homotopy category (cf. \cite{MV}), which can be obtained as a Bousfield localization of $\mathcal{H}_{s} (k)$ (resp. $\mathcal{H}_{s,\bullet} (k)$). We will call objects of $Spc_{k}$ (resp. $Spc_{k,\bullet}$) spaces (resp. pointed spaces).\\
If $\mathcal{X}$ and $\mathcal{Y}$ are spaces, we denote by $[\mathcal{X},\mathcal{Y}]_{\mathbb{A}^{1}} = Hom_{\mathcal{H} (k)} (\mathcal{X},\mathcal{Y})$ the set of morphisms from $\mathcal{X}$ to $\mathcal{Y}$ in $\mathcal{H} (k)$; similarly, if $(\mathcal{X},x)$ and $(\mathcal{Y},y)$ are two pointed spaces, we denote by $[(\mathcal{X},x), (\mathcal{Y},y)]_{\mathbb{A}^{1},\bullet} = Hom_{\mathcal{H}_{\bullet} (k)} ((\mathcal{X},x), (\mathcal{Y},y))$ the set of morphisms from $(\mathcal{X},x)$ to $(\mathcal{Y},y)$ in $\mathcal{H}_{\bullet} (k)$. Sometimes we will omit the basepoints from the notation.\\
The functor $Spc_{k} \rightarrow Spc_{k,\bullet}, \mathcal{X} \mapsto \mathcal{X}_{+} = \mathcal{X} \sqcup \ast$ and the forgetful functor $Spc_{k,\bullet} \rightarrow Spc_{k}$ form a Quillen pair. There are simplicial suspension and loop space functors $\Sigma_{s}: Spc_{k, \bullet} \rightarrow Spc_{k, \bullet}$ and $\Omega_{s}: Spc_{k, \bullet} \rightarrow Spc_{k, \bullet}$, which form an adjoint Quillen pair of functors. The right-derived functor of $\Omega_{s}$ will be denoted $\mathcal{R}\Omega_{s}$. If $(\mathcal{X},x)$ is a pointed space, its simplicial suspension $\Sigma_{s} (\mathcal{X},x) = S^{1} \wedge (\mathcal{X},x)$ has the structure of an $h$-cogroup in $\mathcal{H}_{\bullet} (k)$ (cf. \cite[Definition 2.2.7]{A} or \cite[Section 6.1]{Ho}); in particular, if $(\mathcal{Y},y)$ is another pointed space, there is a natural group structure on the set $[\Sigma_{s}(\mathcal{X},x), (\mathcal{Y},y)]_{\mathbb{A}^{1}, \bullet}$ induced by the $h$-cogroup structure of $\Sigma_{s} (\mathcal{X},x)$. If $(\mathcal{Y},y)$ is any pointed space, the space $\mathcal{R}\Omega_{s}(\mathcal{Y},y)$ has the structure of an $h$-group (or grouplike $H$-space in some literature) in $\mathcal{H}_{\bullet} (k)$ and hence the set $[(\mathcal{X},x), \mathcal{R}\Omega_{s}(\mathcal{Y},y)]_{\mathbb{A}^{1},\bullet}$ has a natural group structure for any pointed space $(\mathcal{X},x)$ induced by the $h$-group structure of $\mathcal{R}\Omega_{s}(\mathcal{Y},y)$.\\
For all $n \geq 1$, let $Q_{2n-1} = Spec (k[x_{1},...,x_{n},y_{1},...,y_{n}]/\langle \sum_{i=1}^{n} x_{i}y_{i} - 1 \rangle)$ be the smooth affine quadric hypersurfaces in $\mathbb{A}^{2n}$. The projection on the coefficients $x_{1},...,x_{n}$ induces a morphism of schemes $p_{2n-1}:Q_{2n-1} \rightarrow \mathbb{A}^{n}\setminus 0$ which is locally trivial with fibers isomorphic to $\mathbb{A}^{n-1}$ and hence an $\mathbb{A}^{1}$-weak equivalence. Hence we have a pointed $\mathbb{A}^{1}$-weak equivalence

\begin{center}
$\mathbb{A}^{n}\setminus 0 \simeq_{\mathbb{A}^{1}} Q_{2n-1}$
\end{center}

for all $n \geq 1$ if we equip $\mathbb{A}^{n}\setminus 0$ with $(1,0,..,0)$ and $Q_{2n-1}$ with $(1,0,..,0,1,0,..,0)$ as basepoints. Now let $R$ be an affine $k$-algebra and $X = Spec(R)$. It follows immediately from the definitions that

\begin{center}
$\mathit{Um}_{n} (R) \cong Hom_{Sch_{k}} (X, \mathbb{A}^{n}\setminus 0)$
\end{center}

and

\begin{center}
$\{(a,b)|a,b \in \mathit{Um}_{n} (R), a b^{t} = 1\} = Hom_{Sch_{k}} (X, Q_{2n-1})$,
\end{center}

where $Sch_k$ is the category of Noetherian $k$-schemes of finite Krull dimension. If $R$ is smooth, $k$ perfect with $char(k) \neq 2$ and $n \geq 3$, it follows from \cite[Remark 7.10]{Mo} and \cite[Theorem 2.1]{F} that in fact

\begin{center}
$\mathit{Um}_{n} (R)/E_{n} (R) \cong [X, \mathbb{A}^{n}\setminus 0]_{\mathbb{A}^{1}}$.
\end{center}

It is well-known that $\mathbb{A}^{n}\setminus 0$ is isomorphic to $\Sigma_{s}^{n-1} \mathbb{G}_{m}^{\wedge n}$ in $\mathcal{H}_{\bullet} (k)$ for all $n \geq 1$; therefore $\mathbb{A}^{n}\setminus 0$ and also $Q_{2n-1}$ inherit the structure of an $h$-cogroup in $\mathcal{H}_{\bullet} (k)$ for $n \geq 2$ (cf. \cite[Definition 2.2.7]{A} or \cite[Section 6.1]{Ho}). In particular, the set $[\mathbb{A}^{n} \setminus 0, \mathbb{A}^{n} \setminus 0]_{\mathbb{A}^{1},\bullet}$ has an induced group structure for $n \geq 2$. For $n \geq 3$, it is well-known that there is a group isomorphism $[\mathbb{A}^{n} \setminus 0, \mathbb{A}^{n} \setminus 0]_{\mathbb{A}^{1},\bullet} \cong GW (k)$ between this group and the Grothendieck-Witt ring of non-degenerate symmetric bilinear forms (cf. \cite[Corollary 5.43]{Mo}).\\
Now let $R$ be a commutative ring, let $n \geq 1$ and, moreover, let $a = (a_{1},...,a_{n})$, $b = (b_{1},...,b_{n})$ be row vectors of length $n$ over $R$. In \cite{S2} Suslin inductively defined matrices $\alpha_{n} (a,b)$ of size $2^{n-1}$ called Suslin matrices for all $n \geq 1$:\\
For $n=1$, one simply sets $\alpha_{1} (a,b) = (a_{1})$; for $n \geq 2$, one sets $a'=(a_{2},...,a_{n})$, $b'=(b_{2},...,b_{n})$ and then defines

\begin{center}
$\alpha_{n} (a,b) = \begin{pmatrix}
a_{1} {Id}_{2^{n-2}} & \alpha_{n-1} (a',b') \\
-{\alpha_{n-1} (b',a')}^{t} & b_{1} {Id}_{2^{n-2}}
\end{pmatrix}.$
\end{center}

In \cite[Lemma 5.1]{S2} Suslin proved the formula $\det (\alpha_{n} (a,b)) = {(a b^{t})}^{2^{n-2}}$ for all $n \geq 2$; in particular, if $a = (a_{1},...,a_{n}) \in Um_{n}(R)$ is a unimodular row and $b=(b_{1},...,b_{n}) \in Um^{t}_{n}(R)$ is a unimodular column such that $\sum_{i=1}^{n} a_{i} b_{i} = 1$ (i.e., $b$ defines a section of $a$), then $\alpha_{n} (a,b) \in SL_{2^{n-1}} (R)$.\\
Suslin constructed these matrices in order to prove that for any unimodular row $a = (a_{1},a_{2},a_{3},...,a_{n})$ of length $n \geq 3$ over $R$, the row of the form $a' = (a_{1},a_{2},a_{3},...,a_{n}^{(n-1)!})$ is completable to an invertible matrix of rank $n$. More precisely, he showed that for any $a$ with section $b$ there exists an invertible $n \times n$-matrix $\beta (a,b)$ whose first row is $a'$ such that the classes of $\beta (a,b)$ and $\alpha_{n} (a,b)$ in $K_{1} (R)$ coincide (cf. \cite[Proposition 2.2 and Corollary 2.5]{S3}).\\
Following \cite{AF}, one can interpret Suslin's construction as a morphism of spaces: Indeed, there exists a morphism $\alpha_{n}: Q_{2n-1} \rightarrow SL_{2^{n-1}}$ induced by $\alpha_{n} (x,y)$, where $x=(x_{1},...,x_{n})$ and $y=(y_{1},...,y_{n})$; if we consider $Q_{2n-1}$ and $SL_{2^{n-1}}$ pointed spaces with basepoints $(1,0,...,0,1,0,...,0)$ and the identity matrix of rank $2^{n-1}$, then this morphism is pointed. Composing with the canonical map $SL_{2^{n-1}} \rightarrow SL$, we obtain a (pointed) morphism $Q_{2n-1} \rightarrow SL$ which we also denote by $\alpha_{n}$.\\
Now let $R$ be an affine algebra over $k$. A unimodular row $a = (a_{1},...,a_{n})$ of length $n$ over $R$ together with a section $b = (b_{1},...,b_{n})$ corresponds to the morphism $Spec (R) \rightarrow Q_{2n-1}, (x_{1},...,x_{n}) \mapsto (a_{1},...,a_{n}), (y_{1},...,y_{n}) \mapsto (b_{1},...,b_{n})$. The Suslin matrix $\alpha_{n}(a,b)$ is then just the pullback of $\alpha_{n}(x,y) \in SL_{2^{n-1}}(Q_{2n-1})$ along this morphism of schemes. Analogously, the class $[\alpha_{n}(a,b)] \in SK_{1}(R)$ is then just the pullback of the class $[\alpha_{n}(x,y)] \in SK_{1}(Q_{2n-1})$ along this morphism of schemes.

\begin{Lem}\label{L2.2}
Let $R$ be an affine algebra over a perfect field $k$ and let $l \geq 1$ and $n \geq 3$. Furthermore, let $a = (a_{1},...,a_{n}) \in Um_{n} (R)$ be a unimodular row with a section $b$ and let $a' = (a_{1},...,a_{n}^{l}) \in Um_{n} (R)$ with any section $b'$. If $-1 \in {(k^{\times})}^{2}$, then $l \cdot [\alpha_{n}(a,b)] = [\alpha_{n}(a',b')] \in SK_{1} (R)$. Moreover, if $l$ is even, then $[\alpha_{n}(a',b')] \in SK_{1} (R)$ is an $\dfrac{l}{2}$-fold multiple of an element in $SK_{1}(R)$.
\end{Lem}

\begin{proof}
Since everything is pulled back from $Q_{2n-1}$ with its universal unimodular row $x=(x_{1},...,x_{n})$ and section $y=(y_{1},...,y_{n})$, we only have to prove the statement for $Q_{2n-1}$, $x$ and $y$.\\
Since $Q_{2n-1}$ is a smooth affine scheme over $k$, we can use the $\mathbb{A}^{1}$-representability of algebraic $K$-theory in $\mathcal{H}_{\bullet}(k)$. Consider the morphism $\Psi^{l}: Q_{2n-1} \rightarrow Q_{2n-1}$ in $\mathcal{H}_{\bullet}(k)$ which corresponds up to canonical pointed $\mathbb{A}^{1}$-weak equivalence to the morphism $\mathbb{A}^{n}\setminus 0 \rightarrow \mathbb{A}^{n} \setminus 0, (x_{1},...,x_{n}) \mapsto (x_{1},...,x_{n}^{l})$.\\
It is well-known that $\Psi^{l}$ corresponds to $l_{\epsilon} \in [Q_{2n-1},Q_{2n-1}]_{\mathbb{A}^{1},\bullet} \cong \mathit{GW}(k)$ with respect to the group structure on $[Q_{2n-1},Q_{2n-1}]_{\mathbb{A}^{1},\bullet}$ induced by the $h$-cogroup structure of $Q_{2n-1} \simeq_{\mathbb{A}^{1}} \Sigma_{s}^{n-1}\mathbb{G}_{m}^{\wedge n}$ in $\mathcal{H}_{\bullet}(k)$ (cf. \cite[Proposition 2.1.9]{AFH}); since $-1 \in {(k^{\times})}^{2}$, $\Psi^{l}$ corresponds to $l_{\epsilon} = l \cdot \langle 1 \rangle \in [Q_{2n-1},Q_{2n-1}]_{\mathbb{A}^{1},\bullet} \cong \mathit{GW}(k)$ and hence to $l \cdot id_{Q_{2n-1}}$ with respect to the group structure on $[Q_{2n-1},Q_{2n-1}]_{\mathbb{A}^{1},\bullet}$ induced by the $h$-cogroup structure of $Q_{2n-1} \simeq_{\mathbb{A}^{1}} \Sigma_{s}^{n-1}\mathbb{G}_{m}^{\wedge n}$ in $\mathcal{H}_{\bullet}(k)$. Therefore the morphism $\alpha_{n} \circ \Psi^{l}$ corresponds to the $l$-fold sum of $\alpha_{n}$ in $[Q_{2n-1},SL]_{\mathbb{A}^{1},\bullet}$ with respect to the group structure on $[Q_{2n-1},SL]_{\mathbb{A}^{1},\bullet}$ induced by $h$-cogroup structure of $Q_{2n-1} \simeq_{\mathbb{A}^{1}} \Sigma_{s}^{n-1}\mathbb{G}_{m}^{\wedge n}$ in $\mathcal{H}_{\bullet}(k)$. The usual Eckmann-Hilton argument implies that $\alpha_{n} \circ \Psi^{l}$ also corresponds to the $l$-fold sum of $\alpha_{n}$ in $[Q_{2n-1},SL]_{\mathbb{A}^{1},\bullet}$ with respect to the group structure on $[Q_{2n-1},SL]_{\mathbb{A}^{1},\bullet}$ induced by $h$-group structure of $SL \simeq_{\mathbb{A}^{1}} \mathcal{R}\Omega_{s}BSL$ in $\mathcal{H}_{\bullet}(k)$ (cf. \cite[Section 4.1]{MV} or \cite[Proposition 5.6]{AE} for the equivalence). Since extending $\alpha_{n} \circ \Psi^{l}$ to a pointed morphism ${(Q_{2n-1})}_{+} \rightarrow SL$ computes the Suslin matrix $\alpha_{n}(x',y')$, this proves the first statement.\\
The second statement follows completely analogously as for even $l$ one has $l_{\epsilon} = \dfrac{l}{2} \cdot h$ with $h = \langle 1 \rangle + \langle -1 \rangle \in GW (k)$.
\end{proof}

\subsection{The generalized Vaserstein symbol}\label{The generalized Vaserstein symbol}\label{2.3}

Let $R$ be a commutative ring. We let $P_{0}$ be a projective $R$-module of rank $2$ with a fixed trivialization $\theta_{0}: R \xrightarrow{\cong} \det (P_{0})$ of its determinant.\\
In \cite{Sy} the author defined a generalized Vaserstein symbol

\begin{center}
$V_{\theta_{0}}: Um (P_{0} \oplus R)/E (P_{0} \oplus R) \rightarrow \tilde{V}(R)$.
\end{center}

associated to $P_{0}$ and $\theta_{0}$. The group $\tilde{V} (R)$ is canonically isomorphic to the elementary symplectic Witt group $W_E (R)$ defined in Section \ref{2.1}.\\
If we let $P_{0} = R^{2}$ be the free $R$-module of rank $2$ and consider the canonical trivialization $\theta_{0}: R \xrightarrow{\cong} \det(R^2), 1 \mapsto e_{1} \wedge e_{2}$, where $e_{1} = (1,0)$ and $e_{2} = (0,1)$, then the generalized Vaserstein symbol associated to $R^2$ and $-\theta_{0}$ coincides with the Vaserstein symbol defined by Suslin and Vaserstein in \cite[\S 5]{SV}.\\
As proven in \cite[Theorem 3.1]{Sy2}, the generalized Vaserstein symbol descends to a map

\begin{center}
$V_{\theta_{0}}: Um (P_{0} \oplus R)/SL (P_{0} \oplus R) \rightarrow \tilde{V}_{SL}(R)$,
\end{center}

where $\tilde{V}_{SL}(R)$ is the quotient of $\tilde{V}(R)$ corresponding to $W_{SL}(R)$ under the isomorphism $W_E (R) \cong \tilde{V} (R)$; hence $\tilde{V}_{SL}(R)$ is canonically isomorphic to $W_{SL}(R)$.\\
Focusing on Noetherian rings of dimension $\leq 4$, the author also gave general criteria for the surjectivity and injectivity of the induced map above (cf. \cite[Theorems 3.2 and 3.6]{Sy2}). For the purposes of this paper, we simply state the following sufficient criterion for its bijectivity which is an immediate consequence of the criteria in \cite{Sy2}:

\begin{Prop}\label{Criterion}
Assume $R$ is a Noetherian ring of dimension $\leq 4$. For $n \geq 3$, we let $P_{n} = P_{0} \oplus R^{n-2}$ and $e_{4} = (0,0,1) \in P_{4}$. Then the induced map $V_{\theta_{0}}: Um (P_{0} \oplus R)/SL (P_{0} \oplus R) \rightarrow \tilde{V}_{SL}(R)$ is bijective if $SL (P_{5})$ acts transitively on $Um (P_{5})$ and $Um^t(P_{4}) = Sp(\chi) e_{4}$ for any non-degenerate alternating form on $P_4$.
\end{Prop}

If $P_{0} = R^2$ is the free $R$-module of rank $2$, this criterion means that $SL_{5}(R)$ acts transitively on $Um_{5}(R)$ and $Sp (\chi)$ acts transitively on $Um_{4}^{t}(R)$ for any invertible alternating matrix $\chi$ of rank $4$ over $R$.

\section{Symplectic orbits of unimodular rows}\label{Symplectic orbits of unimodular rows}\label{3}

We finally prove the main results of this paper in this section. In order to do this, we will use the following implicit results of Suslin on orbits of unimodular rows (cf. \cite[Proof of Theorem 1]{S1} and the proof of \cite[Theorem 2.4]{S4}):

\begin{Thm}\label{T3.1}
Let $R$ be a reduced affine algebra of dimension $d \geq 3$ over an algebraically closed field $k$ and let $l$ be an integer $\geq 1$. Then any unimodular row $a = (a_{1},...,a_{d+1}) \in Um_{d+1}(R)$ of length $d+1$ over $R$ can be transformed via elementary matrices to a unimodular row of length $d+1$ over $R$ of the form $b=(b_{1},...,b_{d+1}^{l})$.
\end{Thm}

\begin{Thm}\label{T3.2}
Let $R$ be a normal affine algebra of dimension $d \geq 3$ over an infinite field $k$ with $char(k) \neq 2$ and $c.d.(k) \leq 1$ and let $l$ be an integer $\geq 1$ such that $\mathit{gcd} (l,\mathit{char}(k)) = 1$. Then any unimodular row $a = (a_{1},...,a_{d+1}) \in Um_{d+1}(R)$ of length $d+1$ over $R$ can be transformed via elementary matrices to a unimodular row of length $d+1$ over $R$ of the form $b = (b_{1},...,b_{d+1}^{l})$.
\end{Thm}

Moreover, we will use the following implicit result of Fasel-Rao-Swan (cf. the proof of \cite[Theorem 7.5]{FRS}):

\begin{Thm}\label{T3.3}
Let $R$ be a normal affine algebra of dimension $d \geq 4$ over an algebraically closed field $k$ with $char(k) \neq 2$ and let $l$ be an integer $\geq 1$ such that $\mathit{gcd} (l,\mathit{char}(k)) = 1$. Then any unimodular row $a = (a_{1},...,a_{d}) \in Um_{d}(R)$ of length $d$ over $R$ can be transformed via elementary matrices to a unimodular row of length $d$ over $R$ of the form $b = (b_{1},...,b_{d}^{l})$.
\end{Thm}

We will also use stability results for $SK_1$ proven by Vaserstein (cf. \cite{V}) and Rao-van-der-Kallen  (cf. \cite[Theorem 3.4]{RvdK}) and by Fasel-Rao-Swan (cf. \cite[Corollary 7.7]{FRS}), which we summarize in the next theorem:

\begin{Thm}\label{T3.4}
Let $R$ be a smooth affine algebra of dimension $d \geq 3$ over a field $k$ with $c.d.(k) \leq 1$ and $d! \in k^{\times}$. Then $SL_{d+1} (R)/E_{d+1}(R) \rightarrow SK_{1} (R)$ is injective. If $k$ is algebraically closed, then also  $SL_{d}(R)/E_d (R) \rightarrow SK_{1} (R)$ is injective.
\end{Thm}

The following lemma will be another ingredient for the proofs of the main results in this paper below:

\begin{Lem}\label{L3.5}
Let $R$ be a commutative ring and let $\chi_{1}$ and $\chi_{2}$ be invertible alternating matrices of rank $2n$ over $R$ such that ${\varphi}^{t} (\chi_{1} \perp \psi_{2}) \varphi = \chi_{2} \perp \psi_{2}$ holds for some $\varphi \in SL_{2n+2}(R)$. Furthermore, let $\chi = \chi_{1} \perp \psi_{2}$. If the equality $Um^{t}_{2n+2}(R) = (E_{2n+2}(R) \cap {Sp} (\chi)) e_{2n+2}$ holds, then one has ${\psi}^{t} \chi_{2} \psi = \chi_{1}$ for some $\psi \in SL_{2n}(R)$ such that $[\psi] = [\varphi] \in K_{1}(R)$.
\end{Lem}

\begin{proof}
Let ${\psi''} e_{2n+2} = \varphi e_{2n+2}$ for some ${\psi''} \in E_{2n+2}(R) \cap {Sp} (\chi)$. Then we set ${\psi'} = {(\psi'')}^{-1} \varphi$. Since ${(\psi')}^{t} (\chi_{1} \perp \psi_{2}) {\psi'} = \chi_{2} \perp \psi_{2}$, the matrix corresponding to the composite $\psi: R^{2n} \xrightarrow{\psi'} R^{2n+2} \rightarrow R^{2n}$ and $\psi'$ satisfy the following conditions:

\begin{itemize}
\item ${\psi'} (e_{2n+2}) = e_{2n+2}$;
\item $\pi_{2n+1, 2n+2} {\psi'} = \pi_{2n+1, 2n+2}$;
\item ${\psi}^{t} \chi_{1} \psi = \chi_{2}$.
\end{itemize}
 
These conditions imply that $\psi$ equals ${\psi'}$ up to an element of $E_{2n+2}(R)$ and hence has determinant $1$ as well. Thus, by construction, $[\psi] = [\varphi]$ and ${\psi}^{t} \chi_{1} \psi = \chi_{2}$.
\end{proof}

In fact, the lemma above was essentially proven by Suslin and Vaserstein (cf. \cite[Lemma 5.6]{SV}), but they did not state in their lemma that the assumption $Um^{t}_{2n+2}(R) = (E_{2n+2}(R) \cap {Sp} (\chi)) e_{2n+2}$ immediately implies the last condition $[\psi] = [\varphi] \in K_{1}(R)$. For the convenience of the reader, we repeated their reasoning above. Now we can finally prove the main results of this paper:

\begin{Thm}\label{T3.6}\label{Trans1}
Let $R$ be a smooth affine algebra of odd dimension $d \geq 3$ over a field $k$ such that $c.d.(k) \leq 1$ and $d! \in k^{\times}$; if $d+1 \equiv 0~mod~4$, furthermore assume that $k$ is perfect. Let $\chi \in A_{d+1}(R)$ be an invertible alternating matrix of rank $d+1$. Then $Sp (\chi)$ acts transitively on $Um_{d+1} (R)$.
\end{Thm}

\begin{proof}
First of all, we may assume that $k$ is infinite: If $k$ is finite, then $E_{d+1} (R)$ acts transitively on the left on $Um_{d+1}^{t}(R)$ and transitively on the right on $Um_{d+1}(R)$ (cf. \cite[Corollary 17.3]{SV} and \cite[Theorem 7.2]{SV}). It follows from \cite[Lemma 5.5]{SV} that $Sp({\chi}^{-1})$ acts transitively on the left on $Um_{d+1}^{t} (R)$. By transposition, this implies that $Sp(\chi)$ acts transitively on the right on $Um_{d+1}(R)$. So let us now assume that $k$ is infinite.\\
Let $a = (a_{1},...,a_{d+1}) \in Um_{d+1}(R)$. Now let us first assume that there is $\varphi \in SL_{d+1}(R)$ with first row $a$ such that its class $[\varphi] \in K_{1} (R)$ lies in the image of the forgetful map $K_{1}{Sp} (R) \xrightarrow{f} K_{1} (R)$. By Lemma \ref{L2.1}, this implies that there is $\psi \in Sp(\chi \perp \psi_{2m})$ for some $m \geq 0$ such that $[\psi]=[\varphi] \in K_1 (R)$. Since $E_{d+1+2n}(R) \cap Sp(\chi \perp \psi_{2n})$ acts transitively on $Um_{d+1+2n} (R)$ for $n \geq 1$ by \cite[Chapter IV, Theorem 3.4]{HB} and \cite[Lemma 5.5]{SV}, it then follows from Lemma \ref{L3.5} that we can assume that $m = 0$.\\
As the homomorphism $SL_{d+1} (R)/E_{d+1} (R) \rightarrow SK_{1} (R)$ is injective in view of the stability results of Rao-van-der-Kallen, it follows that $\varphi {\psi}^{-1} \in E_{d+1} (R)$. Since by \cite[Lemma 5.5]{SV} the equality $\pi_{1,d+1} E_{d+1} (R) = \pi_{1,d+1} ({E}_{d+1} (R) \cap Sp(\chi))$ holds, there is $\psi' \in {E}_{d+1} (R) \cap Sp (\chi)$ such that $\pi_{1,d+1}\varphi {\psi}^{-1} = \pi_{1,d+1} \psi'$. In particular, $a = \pi_{1,d+1} \varphi = \pi_{1,d+1} \psi' \psi$ lies in the orbit of $\pi_{1,d+1}$ under the action of $Sp (\chi)$.\\
Altogether, we have established the following statement in the previous paragraphs: For $a = (a_{1},...,a_{d+1}) \in Um_{d+1}(R)$, it suffices to show that there is $\varphi \in SL_{d+1}(R)$ with first row $a$ such that its class $[\varphi] \in K_{1} (R)$ lies in the image of the forgetful map $K_{1}{Sp} (R) \xrightarrow{f} K_{1} (R)$.\\
But by the implicit results of Suslin, any unimodular row of length $d+1$ can be transformed via elementary matrices to a row of the form $a = (a_{1},...,a_{d+1}^{2{d!}^{2}})$. Thus, by the reduction step in the previous paragraph, it suffices to consider rows of the latter type.\\
So let $a = (a_{1},...,a_{d+1}^{2{d!}^{2}})$ be a row of this type. Furthermore, we let $b'$ be a section of the row $a' = (a_{1},...,a_{d+1}^{2d!})$. The matrix $\beta_{d+1}(a',b') \in SL_{d+1}(R)$ is a completion of $a$. If $d+1 \equiv 2~mod~4$, it is well-known that $[\beta_{d+1}(a',b')]$ lies in the image of the forgetful map $K_{1}{Sp} (R) \xrightarrow{f} K_{1} (R)$ (cf. \cite[Proposition 3.3.3]{AF}); this proves the theorem in case $d+1 \equiv 2~mod~4$.\\
So let us henceforth assume that $d+1 \equiv 0~mod~4$ holds. In this case, one knows that $W_E (Q_{2(d+1)-1}) = \mathbb{Z}/2\mathbb{Z}$ by \cite[Proposition 2.7]{Sy2} and Lemma \ref{L2.2} shows that $[\beta_{d+1} (x',y')] \in SK_{1}(Q_{2(d+1)-1})$ is a $d!$-fold multiple of an element in $SK_{1}(Q_{2(d+1)-1})$, where $x'=(x_{1},...,x_{d},x_{d+1}^{2d!})$ and $y'$ denotes a section of $x'$; as a matter of fact, if $-1 \in {(k^{\times})}^{2}$, Lemma \ref{L2.2} even shows that $2d! \cdot [\beta_{d+1} (x,y)] = [\beta_{d+1} (x',y')] \in SK_{1}(Q_{2(d+1)-1})$, where $x=(x_{1},...,x_{d+1})$ and $y=(y_{1},...,y_{d+1})$. But since $d!$ is even and $W_E (Q_{2(d+1)-1})$ is $2$-torsion, one obtains $H([\beta_{d+1}(x',y')]) = 0$. Now let $a''=(a_{1},...,a_{d+1}) \in Um_{d+1}(R)$ with section $b''$ and consider the morphism $\varphi: Spec(R) \rightarrow Q_{2(d+1)-1}, (x,y) \mapsto (a'',b'')$. The diagram
\begin{center}
$\begin{xy}
  \xymatrix{
      SK_1 (Q_{2(d+1)-1}) \ar[r]^{H} \ar[d]_{\varphi^{\ast}}    &   W_E (Q_{2(d+1)-1}) \ar[d]^{\varphi^{\ast}}  \\
      SK_{1} (R) \ar[r]_{H}             &   W_E (R)   
  }
\end{xy}$
\end{center}
is obviously commutative and the homomorphism $\varphi^{\ast}: SK_{1} (Q_{2(d+1)-1}) \rightarrow SK_{1}(R)$ sends $[\beta_{d+1} (x',y')]$ to $[\beta_{d+1} (a',b')]$. Hence pulling back the equality $H([\beta_{d+1}(x',y')]) = 0 \in W_E (Q_{2(d+1)-1})$ along the morphism $\varphi: Spec(R) \rightarrow Q_{2(d+1)-1}, (x,y) \mapsto (a'',b'')$, proves that $H([\beta_{d+1}(a',b')]) = 0 \in W_E (R)$. By exactness of the Karoubi periodicity sequence, this implies that $[\beta_{d+1}(a',b')]$ lies in the image of the forgetful map $K_{1}{Sp} (R) \xrightarrow{f} K_{1} (R)$; this proves the theorem in case $d+1 \equiv 0~mod~4$.
\end{proof}

\begin{Rem}\label{R3.7}
The proof of the previous theorem would work for a reduced (not necessarily smooth) affine algebra $R$ of odd dimension $d \geq 3$ over an algebraically closed field $k$ with $d! \in k^{\times}$ as well if the stability result for $SK_1$, i.e., the injectivity of $SL_{d+1}(R)/E_{d+1}(R) \rightarrow SK_{1}(R)$, was known for such general algebras.
\end{Rem}

\begin{Thm}\label{T3.8}\label{Trans2}
Let $R$ be a smooth affine algebra of even dimension $d \geq 4$ over an algebraically closed field $k$ with $d! \in k^{\times}$. Let $\chi \in A_{d}(R)$ be an invertible alternating matrix of rank $d$. Then $Sp (\chi)$ acts transitively on $Um_{d} (R)$.
\end{Thm}

\begin{proof}
The proof is very analogous to the proof of the previous theorem. For the convenience of the reader, we give a precise proof with all necessary adjustments to the previous one:\\
Let $a = (a_{1},...,a_{d}) \in Um_{d}(R)$. First assume as in the previous proof that there is $\varphi \in SL_{d}(R)$ with first row $a$ such that its class $[\varphi] \in K_{1} (R)$ lies in the image of the forgetful map $K_{1}{Sp} (R) \xrightarrow{f} K_{1} (R)$. Again by Lemma \ref{L2.1}, there is $\psi \in Sp(\chi \perp \psi_{2m})$ for some $m \geq 0$ such that $[\psi]=[\varphi] \in K_1 (R)$. But $E_{d+2n}(R) \cap Sp(\chi \perp \psi_{2n})$ acts transitively on $Um_{d+2n} (R)$ for $n \geq 1$ by \cite[Chapter IV, Theorem 3.4]{HB} and \cite[Lemma 5.5]{SV}, so we can assume that $m = 0$ in view of Lemma \ref{L3.5}.\\
The stability result of Fasel-Rao-Swan asserts that $SL_{d} (R)/E_{d} (R) \rightarrow SK_{1} (R)$ is injective. Therefore $\varphi {\psi}^{-1} \in E_{d} (R)$. Since $\pi_{1,d} E_{d} (R) = \pi_{1,d} ({E}_{d} (R) \cap Sp(\chi))$ holds in view of \cite[Lemma 5.5]{SV}, there is $\psi' \in {E}_{d} (R) \cap Sp (\chi)$ such that $\pi_{1,d}\varphi {\psi}^{-1} = \pi_{1,d} \psi'$. As a consequence, $a = \pi_{1,d} \varphi = \pi_{1,d} \psi' \psi$ lies in the orbit of $\pi_{1,d}$ under the right action of $Sp (\chi)$.\\
Altogether, we have again established the following reduction: For any unimodular row $a = (a_{1},...,a_{d}) \in Um_{d}(R)$, it suffices to show that there is $\varphi \in SL_{d}(R)$ with first row $a$ such that its class $[\varphi] \in K_{1} (R)$ lies in the image of the forgetful map $K_{1}{Sp} (R) \xrightarrow{f} K_{1} (R)$.\\
The implicit result of Fasel-Rao-Swan (cf. Theorem \ref{T3.3}) asserts that any unimodular row of length $d$ can be transformed via elementary matrices to a row of the form $a = (a_{1},...,a_{d}^{{d!}^{2}})$. As in the proof of the previous theorem, it hence suffices to consider rows of the latter type.\\
So let again $a = (a_{1},...,a_{d}^{{d!}^{2}})$ be a row of this type. Moreover, we let $b'$ be a section of the row $a' = (a_{1},...,a_{d}^{d!})$. The matrix $\beta_{d}(a',b') \in SL_{d}(R)$ is a completion of $a$. If $d \equiv 2~mod~4$, it is well-known that $[\beta_{d}(a',b')]$ lies in the image of the forgetful map $K_{1}{Sp} (R) \xrightarrow{f} K_{1} (R)$ (cf. \cite[Proposition 3.3.3]{AF}); this proves the theorem in case $d \equiv 2~mod~4$.\\
Therefore we can assume that $d \equiv 0~mod~4$, i.e., $d$ is divisible by $4$. Again, one knows that $W_E (Q_{2d-1}) \cong \mathbb{Z}/2\mathbb{Z}$ in this case. Furthermore, Lemma \ref{L2.2} shows that $d! \cdot [\beta_{d} (a'',b'')] = [\beta_{d} (a',b')] \in K_{1}(R)$, where $a''=(a_{1},...,a_{d})$ and $b''$ is a section of $a''$. But since $d!$ is even and $W_E (Q_{2d-1})$ is $2$-torsion, one obtains $d! \cdot H([\beta_{d}(x,y)]) = 0$. Hence pulling back this equality along the morphism $Spec(R) \rightarrow Q_{2d-1}, (x,y) \mapsto (a'',b'')$, yields $H([\beta_{d}(a',b')]) = 0 \in W_E (R)$. The exactness of the Karoubi periodicity sequence implies again that $[\beta_{d}(a',b')]$ lies in the image of the forgetful map $K_{1}{Sp} (R) \xrightarrow{f} K_{1} (R)$; this finishes the proof of the theorem in case $d \equiv 0~mod~4$.
\end{proof}

\begin{Thm}\label{T3.9}\label{Bij}
Let $R$ be a smooth affine algebra of dimension $d \geq 3$ over an infinite field $k$ with $d! \in k^{\times}$. Let $\theta_{0}: R \xrightarrow{\cong} \det (R^2)$ be a fixed trivialization of $\det (R^2)$. Assume that one of the following conditions is satisfied:
\begin{itemize}
\item $d=3$, $k$ is perfect and $c.d.(k) \leq 1$
\item $d=4$ and $k$ is algebraically closed
\end{itemize}
Then $V_{\theta_{0}}: Um_{3}(R)/SL_{3}(R) \xrightarrow{\cong} \tilde{V}_{SL}(R)$ is a bijection.
\end{Thm}

\begin{proof}
We verify the sufficient criterion given by Proposition \ref{Criterion}. By \cite[Chapter IV, Theorem 3.4]{HB} and \cite[Theorem 1]{S1} it is clear that $SL_{5}(R)$ acts transitively on $Um_{5}(R)$. The previous theorems establish in particular that $Sp (\chi^{-1})$ acts transitively on the right on $Um_{4}(R)$ for any invertible alternating matrix $\chi$ of rank $4$; by transposition, it follows that $Sp(\chi)$ acts transitively on the left on $Um_{4}^{t}(R)$. Hence the bijectivity follows immediately from Proposition \ref{Criterion}.
\end{proof}

\begin{Rem}\label{R3.10}
If one could prove the statement in Theorem \ref{T3.6} for an arbitrary reduced (not necessarily smooth) affine algebra $R$ of dimension $3$ over an algebraically closed field $k$ with $6 \in k^{\times}$, one would immediately obtain that the statement in Theorem \ref{T3.9} holds for such algebras as well. The consequence would be the following: For any reduced affine (not necessarily smooth) algebra $R$ of dimension $3$ over an algebraically closed field, the set of stably free oriented $R$-modules of rank $2$ would correspond to the group $W_{SL}(R)$. In particular, all stably free $R$-modules of rank $2$ would be free if and only if $W_{SL}(R)$ would be trivial.
\end{Rem}

\begin{Thm}\label{T3.11}
Let $R$ be an affine algebra of dimension $d=3$ over a finite field $\mathbb{F}_{q}$. Let $P_{0}$ be a projective $R$-module of rank $2$ with a trivial determinant and let $P_{n} = P_{0} \oplus R^{n-2}$ for $n \geq 3$ and $e_{4} = (0,0,1) \in P_4$. Then the equality $Sp (\chi) e_{4} = Um^t (P_{4})$ holds for any non-degenerate alternating form $\chi$ on $P_{4}$. In particular, the generalized Vaserstein symbol associated to any trivialization $\theta_{0}$ of $\det (P_{0})$ gives a bijection $V_{\theta_{0}}: Um (P_{0} \oplus R)/SL (P_{0} \oplus R) \xrightarrow{\cong} \tilde{V}_{SL} (R)$.
\end{Thm}

\begin{proof}
By \cite[Chapter IV, Theorem 3.4]{HB}, the group $E(P_{5})$ and hence $SL (P_{5})$ acts transitively on $Um (P_{5})$. Moreover, the group $E(P_{4})$ acts transitively on $Um^t (P_{4})$ by \cite[Corollary 17.3]{SV} and \cite[Theorem 1.1]{DK}. The first statement then follows from \cite[Lemma 2.8]{Sy}. Then Proposition \ref{Criterion} implies the second statement.
\end{proof}

We remark that the group $\tilde{V}_{SL} (R)$ is not trivial in general for an affine algebra of dimension $3$ over a finite field: Let $\mathbb{F}_{p}$ be the field with $p$ elements for a prime number $p$ with $p \equiv 1~mod~8$. We consider the polynomial $X^{8} - a$ for some element $a \in \mathbb{F}_{p}^{\times}$ which is not a square; furthermore, we let $\sqrt[8]{a}$ be a root of this polynomial in an algebraic closure of $\mathbb{F}_{p}$. Since $p \equiv 1~mod~8$, the field $\mathbb{F}_{p}$ contains all $8$th roots of unity, i.e., all zeros of the polynomial $X^{8} - 1$ over $\mathbb{F}_{p}$. In particular, by Kummer's theorem on cyclic field extensions, we see that $\mathbb{F}_{p}(\sqrt[8]{a})$ is Galois over $\mathbb{F}_{p}$ and $[\mathbb{F}_{p}(\sqrt[8]{a}):\mathbb{F}_{p}] = r$ such that $r$ divides $8$. Therefore the minimal polynomial $\mathcal{M} (\sqrt[8]{a})$ of $\sqrt[8]{a}$ over $\mathbb{F}_{p}$ has degree $1,2,4$ or $8$. But the coefficient of $\mathcal{M} (\sqrt[8]{a})$ in degree zero is a product of $8$th roots of unity (which are all in $\mathbb{F}_{p}$) and ${\sqrt[8]{a}}^{r}$. Since $a$ is not a square in $\mathbb{F}_{p}$, it follows that ${\sqrt[8]{a}}^{i} \notin \mathbb{F}_{p}$ for $i =1,2,4$ and $r$ has to be $8$. Hence $X^{8} - a$ is irreducible over $\mathbb{F}_{p}$. If we take the polynomial $X^{2} - a$ for N. Mohan Kumar's construction of stably free modules in \cite{NMK}, then we produce a smooth affine algebra $R_{\mathcal{K}}$ of dimension $3$ over $\mathbb{F}_{p}$ which admits a non-free stably free module of rank $2$. It follows from the previous theorem that $\tilde{V}_{SL} (R_{\mathcal{K}}) \neq 0$.

\end{document}